\documentclass[11pt,a4paper]{article}
\usepackage{a4,amssymb,amsmath,amsthm,url,color}
\usepackage{bezier,amsfonts,amssymb,graphicx,amsthm,url}
\usepackage[english]{babel}
\title{$(2,2)$-colourings and clique-free $\sigma$-hypergraphs}
\date{}
\begin{document}
\newtheorem{theorem}{Theorem}[section]
\newtheorem{definition}{Definition}[section]
\newtheorem{proposition}[theorem]{Proposition}
\newtheorem{corollary}[theorem]{Corollary}
\newtheorem{lemma}[theorem]{Lemma}

\author{Yair Caro \\ Department of Mathematics\\ University of Haifa-Oranim \\ Israel \and Josef  Lauri\\ Department of Mathematics \\ University of Malta
\\ Malta \and Christina Zarb \\Department of Mathematics \\University of Malta \\Malta }

\maketitle

\begin{abstract}
We consider vertex colourings of $r$-uniform hypergraphs $H$ in the classical sense, that is such that no edge has all its vertices given the same colour, and $(2,2)$-colourings of $H$ in which the vertices in any edge are given exactly two colours.  This is a special case of constrained colourings introduced by Bujtas and Tuza which, in turn, is a generalisation of Voloshin's colourings of mixed hypergraphs.  We study, $\chi(H)$, the classical chromatic number, and the $(2,2)$-spectrum of $H$, that is, the set of integers $k$ for which $H$ has a $(2,2)$-colouring using exactly $k$ colours.

We present extensions of hypergraphs which preserve both the chromatic number and the $(2,2)$-spectrum and which, however often repeated, do not increase the clique number of $H$ by more than a fixed number. In particular, we present sparse $(2,2)$-colourable clique-free $\sigma$-hypergraphs having arbitrarily large chromatic number - these $r$-uniform hypergraphs were studied by the authors in earlier papers.   We use these ideas to extend some known $3$-uniform hypergraphs which exhibit a $(2,2)$-spectrum with remarkable gaps.  We believe that this work is the first to present an extension of hypergraphs which preserves both $\chi(H)$ and the $(2,2)$-spectrum of $H$ simultaneously.
\end{abstract}
\section{Introduction}

Let $V=\{v_1,v_2,...,v_N\}$ be a finite set, and let $E=\{E_1,E_2,...,E_m\}$ be a family of subsets of $X$.  The pair $H=(X,E)$ is called a \emph{hypergraph} with vertex- set $V(H)=V$, and with edge-set $E(H)=E$.  When all the subsets are of the same size $r$, we say that $H$ is an \emph{r-uniform hypergraph}, and if all possible such subsets of size $r$ are taken to be edges, $H$ is said to be a \emph{complete $r$-uniform hypergraph}, or a \emph{clique}.  A clique on $N$ vertices is also referred to as an $N$-clique.  The maximum number of vertices of a clique in $H$ is called the \emph{clique number} and is denoted by $\omega(H)$.  A clique in $H$ is therefore a subset of $V(H)$ which forms a complete $r$-uniform subhypergraph in $H$.  

A \emph{constrained colouring}, or $(\alpha,\beta)$-colouring of a hypergraph $H$, is an assignment of colours to its vertices such that no edge of $H$ contains less than $\alpha$ or more than $\beta$ vertices with different colours; this notion was first introduced in \cite{bujtastuz09}, and studied further in a number of publications.   It can be considered as an extension of Voloshin colourings, also called mixed hypergraphs - the book \cite{voloshin02} and the up-to-date website \url{http://http://spectrum.troy.edu/voloshin/publishe.html} are recommended sources for literature on all these types of colourings of hypergraphs.   An $(\alpha,\beta)$-colouring which uses exactly $k$ colours is said to be a $k$-$(\alpha,\beta)$-colouring. The \emph{lower chromatic number} $\chi_{\alpha,\beta}$ of $H$ is defined as the least number $k$ for which $H$ has a $k$-$(\alpha,\beta)$-colouring.  Similarly, the \emph{upper chromatic number} $\overline{\chi}_{\alpha,\beta}$ is the largest $k$ for which $H$ has a $k$-$(\alpha,\beta)$-colouring.   The \emph{$(\alpha,\beta)$-spectrum} of $H$, $Spec_{(\alpha,\beta)}(H)$,  is the sequence, in increasing order, of all $k$ such that $H$ has a $k$-$(\alpha,\beta)$-colouring.  Clearly, the first and last terms of this sequence are $\chi_{\alpha,\beta}$ and $\overline{\chi}_{\alpha,\beta}$ respectively.  We say that the $(\alpha,\beta)$-spectrum has a \emph{gap}, or is \emph{broken}, when there exist integers $k_1<k_2<k_3$ such that the hypergraph is $k_1$-$(\alpha,\beta)$- and $k_3$-$(\alpha,\beta)$-colourable but not $k_2$-$(\alpha,\beta)$-colourable.  The smallest hypergraph with broken chromatic spectrum was constructed in \cite{jiang1999}.   The works of \cite{colburn99,Gionfriddo04,gionfriddo04bicolouring} study the spectra of hypergaphs obtained from design theory.  A \emph{classical colouring} of an $r$-uniform hypergraph is a $(2,r)$-colouring, that is a colouring in which no edge is monochromatic.  The \emph{chromatic number} of a hypergraph $H$, $\chi(H)$ is the minimum number of colours required for a  $(2,r)$-colouring of $H$.  In this paper we mainly consider classical colourings and $(2,2)$-colourings, that is colourings in which every edge contains exactly two colours.  

The main aim of this paper is to develop a technique by which we can augment an $r$-uniform hypergraph, that is add vertices and edges, while preserving certain properties and parameters, in the spirit envisioned in \cite{voloshin1999pseudo}.  In particular, we are interested in  increasing the number of vertices and edges to produce new $r$-uniform hypergraphs while preserving the chromatic number $\chi$, and the $(2,2)$-chromatic spectrum.  We present such a technique which we call the $(p,q)$-extension, and show that when it is applied to a $r$-uniform hypergraph $H$, both the chromatic number and the $(2,2)$-spectrum are unchanged.  To the best of our knowledge this is the first instance of an extension of an $r$-uniform hypergraph preserving both the chromatic number and the $(2,2)$-spectrum.

We then study the clique number of $\sigma$-hypergraphs.  A $\sigma$-hypergraph $H= H(n,r,q$ $\mid$ $\sigma$), where $\sigma$ is a partition of $r$,  is an $r$-uniform hypergraph having $nq$ vertices partitioned into $ n$ \emph{classes} of $q$ vertices each.  If the classes are denoted by $V_1$, $V_2$,...,$V_n$, then a subset $K$ of $V(H)$ of size $r$ is an edge if the partition of $r$ formed by the non-zero cardinalities $ \mid$ $K$ $\cap$ $V_i$ $\mid$, $ 1 \leq i \leq n$, is $\sigma$. The non-empty intersections $K$ $\cap$ $V_i$ are called the parts of $K$, and $s(\sigma)$ denotes the number of parts. These hypergraphs were introduced in \cite{CaroLauri14} and further studied in \cite{CLZ1}, and they also present interesting $(\alpha,\beta)$-spectra, in particular when $\alpha=\beta=2$.  We show in this paper that there exist $\sigma$-hypergraphs with low clique number and arbitrarily high chromatic number, in the spirit of the literature about sparse  hypergraphs with large chromatic number such as \cite{erdos1975problems,Gebauer20131483,kostochka2010}. 

In the last section, we look at applications of the $(p,q)$-extension.  In particular, we look at a specific instance of this extension which we call the \emph{$t$-star extension}. This extension is applied to the hypergraphs constructed in \cite{Gionfriddo04}, which presents $3$-uniform hypergraphs with various very interesting $(2,2)$-spectra.  We also consider this extension applied to $\sigma$-hypergraphs.   

In the case of $\sigma$-hypergraphs, we also consider the effect of the $t$-star extension on the clique number.  We show that when the extension under certain restrictions is repeated several times , giving $r$-uniform hypergraphs with large numbers of vertices and edges, the clique number and chromatic number remain stable. 

\section{Extensions of $r$-uniform hypergraphs}

In this section we present a method which extends an $r$-uniform hypergraph $H$ using an operation that simultaneously preserves  the chromatic number and the $(2,2)$-spectrum.

Consider an $r$-uniform hypergraph $H$, $r \geq 3$.  Let $E^*=\{v_1,v_2, \ldots, v_r\}$ be an edge of $H$.   Let $W=\{w_1,w_2,\ldots,w_p\}$ and $U=\{u_1,u_2,\ldots,u_q\}$ such that $p \geq 1$ and $q \geq 0$.  Let $T$ be a non-empty subset of   $\{1,2,\ldots, \min\{p, \lfloor\frac{r-1}{2} \rfloor \}\}$.  Let $P$ be a non-empty subset of $\{1,2,\ldots,\min\{p,r-2\}\}$ and let $Q$ be  a non-empty subset of $\{1,2,\ldots,\min\{q,r-2\}\}$ such that $\exists$ $x \in P$ and $\exists$ $y \in Q$ such that $x+y \leq r-1$.   Consider the new $r$-uniform hypergraph $H(p,q,T,P,Q)$ formed as follows: the vertex set of  of $H(p,q,T,P,Q)$ consists of the union of $V(H)$,$W$ and $U$.  All edges of $H$ are edges of $H(p,q,T,P,Q)$, together with new edges formed as follows:  

\begin{itemize}
\item{\emph{Type 1 edges} :   An $r$-set $K \subseteq E^* \cup W$ is a Type 1 edge if $|K \cap W| \in T$.}
\item{\emph{Type 2 edges}:  An $r$-set $K \subseteq E^* \cup W \cup U$ is a Type 2 edge if $|K \cap W| \in P$, $|K \cap U| \in Q$ such that $|K \cap W| +|K \cap U| \leq r-1$, and $|K \cap E^*|= r - |K\cap W \cup U|.$  These edges exist if $U$ is not the empty set and by the existence of $x \in P$ and $y \in Q$ such that $x+y \leq r-1$.}
\end{itemize}

For short we call this a \emph{$(p,q)$-extension} of $H$ and denote it by $H_{p,q}$.

\begin{theorem} \label{chrom_no}
Consider an $r$-uniform hypergraph $H$, $r \geq 3$, with chromatic number $\chi(H)$.  Then $\chi(H_{p,q})=\chi(H)$.
\end{theorem}
\begin{proof}

Let $\chi(H)=k$ and consider $H_{p,q}$.  Since $H \subseteq H_{p,q}$, $\chi(H_{p,q}) \geq \chi(H)$.

Consider a $k$-colouring of $H$.  We show that this can be extended to a $k$-colouring of $H_{p,q}$.  Consider $E^*=\{v_1,v_2,\ldots,v_r\}$, the edge used for the extension - the vertices of this edge are coloured by some $s$ colours, where $2 \leq s \leq \min\{k,r\}$.  Let these colours be $1,2,\ldots,s$, and let $c_i$ be the number of vertices coloured $i$, where $1 \leq i \leq s$.  Without loss of generality, assume $c_1 \leq c_2 \leq \ldots \leq c_s$.  Then $c_1 \leq \lfloor \frac{r}{2} \rfloor$.  Now consider the vertices $W=\{w_1,w_2,\ldots,w_p\}$ and $U=\{u_1,u_2,\ldots,u_q\}$ .  Let us give these vertices in $W$ the colour $1$, and the vertices in $U$ the colour 2.  We show that this is a valid $k$-colouring of $H^*_{p,q}$.

Any edge in $H_{p,q}$ which is unchanged from $H$ keeps the same colouring, and so has a valid nonmonochromatic colouring.  Now let us consider a type 1 edge:  this has at most $\lfloor \frac{r-1}{2} \rfloor +  \lfloor \frac{r}{2} \rfloor$ vertices which are coloured 1.  But $\lfloor \frac{r-1}{2} \rfloor +  \lfloor \frac{r}{2} \rfloor \leq  \frac{r-1}{2}  + \frac{r}{2} \leq \frac{2r-1}{2} < r$ and hence there is at least one other vertex in the edge which must have a colour different from 1.  Hence type 1 edges are not monochromatic.

If type 2 edges exist, then they have at least one vertex taken from $W$ and at least one vertex taken from $U$: these two vertices are coloured 1 and 2 respectively, and hence again a type 2 edge is not monochromatic. Therefore this is a valid $k$-colouring of $H_{p,q}$ and therefore $\chi(H_{p,q})=\chi(H)=k$.
\end{proof}

\begin{theorem}
Let $H$ be an $r$-uniform hypergraph, $r \geq 3$.  Then $Spec_{(2,2)}(H_{p,q}) = Spec_{(2,2)}(H)$.
\end{theorem}

\begin{proof}

Suppose $k \not \in Spec_{(2,2)}(H)$ and let us consider a $k$-$(2,2)$-colouring of $H_{p,q}$.  Since $H$ is not $k$-$(2,2)$-colourable,  there must be $1 \leq i \leq \min\{k-2,r\}$ colours which do not appear on the vertices of $H$, and hence must appear on the vertices in $W$ and  $U$, while $k-i$ colours are used on the vertices of $H$.   Let us assume that for some value of $i$ between 1 and $\min\{k-2,r\}$, $H$ is $(k-i)$-$(2,2)$-colourable (otherwise we are done and $k \not \in Spec_{(2,2)}(H_{p,q}))$.  We colour $H$ using these $k-i$ colours and use the remaining $i$ colours on the new vertices of $H_{p,q}$.  Hence at least one vertex in $W \cup U$ has a colour different from those used in $H$.  We also note that the vertices of $E^*$ must be coloured using exactly two colours, say 1 and 2.

So  suppose there exists $x \in W$ which has a colour which does not appear on the vertices of $H$.  Recall that we can take at most $\lfloor \frac{r-1}{2} \rfloor$ vertices from $W$ for a Type 1 edge, and $\lfloor \frac{r-1}{2} \rfloor \leq  \frac{r-1}{2} \leq r-2$ for $r\geq 3$.  Hence any Type 1 edge must contain at least two vertices taken from $E^*$.  So consider the Type 1 edge which includes two vertices taken from $E^*$ which have different colours, the vertex $x$, and the remaining $r-3$ vertices taken from $E^*$ and $W$ as required (by $T$) - this is a valid Type 1 edge which contains at least three colours and hence it makes the $k$-$(2,2)$-colouring of $H_{p,q}$ invalid.

Now consider the case where all vertices in $W$ have colours which appear in $H$ - then at least one vertex in $x \in U$ contains a colour not used in $H$.  A Type 2 edge must contain at least one vertex from each of the three sets $E^*$, $W$ and $U$.  Since the vertices in $E^*$ use the colours 1 and 2,  we can always choose a vertex from $W$ and a vertex from $E^*$ which have different colours, say, without loss of generality vertex $v_1 \in E^*$ and vertex $w_1 \in W$ have different colours.  Consider the Type 2 edge which includes the vertices $v_1, w_1$ and $x$ with the remaining $r-3$ vertices taken from $E^*$,$W$ and $U$ as required (by $P$ and $Q$) - again this is a valid type 2 edge which contains 3 colours hence making the $k$-$(2,2)$-colouring of $H_{p,q}$ invalid.

Let us now consider $k \in Spec_{(2,2)}(H)$, and a valid $k$-$(2,2)$-colouring of $H$.  Let us extend this colouring to $H_{p,q}$ by considering two cases.

First consider the two colours which appear in $E^*$, say colours 1 and 2.  If one of the colour appears on more vertices than the other,  say colour 2 appears more than colour 1, we call 2 the majority colour and 1 the minority colour.  In this case we give the vertices in $W$ the minority colour 1, while the vertices in $U$ (if any) receive the majority colour 2.  Thus, in $E^*$ there is at least one and at most $\lfloor \frac{r-1}{2} \rfloor$ vertices which receive the colour 1, and at least two and at most $r-1$ vertices which receive the colour 2.  If we consider a Type 1 edge, then the minority colour appears on at least one and at most $\lfloor \frac{r-1}{2} \rfloor + \lfloor \frac{r-1}{2} \rfloor \leq 2 \left(\frac{r-1}{2} \right ) \leq r-1$ vertices, and hence there must be at least one and at most $r-1$ vertices which must be taken from $E^*$ having the colour 2.  Hence any type 1 edge has exactly 2 colours.  If we consider a type 2 edge this must have at least one vertex from $W$ and at least one vertex from $U$ which have different colours and hence the edge is not monochromatic - all other vertices in the edge are taken from $W$, $U$ and $E^*$ and hence must have colours 1 or 2.  Therefore any type 2 edge has exactly two colours too.

Now let us consider the case in which both colours in $E^*$ appear an equal number of times (hence $r$ is even).  Then we can give the vertices in $W$ one colour, say colour 1, while the vertices in $U$ receive colour 2.  A type 1 edge has at least one and at most $\lfloor \frac{r-1}{2} \rfloor +\frac{r}{2}$ vertices which receive the colour 1.  But  $\lfloor \frac{r-1}{2} \rfloor +\frac{r}{2} \leq \frac{r-2}{2} + \frac{r}{2} \leq \frac{2r-2}{2} \leq r-1$ - hence it must have at least one and at most $r-1$ other vertices taken from $E^*$ which must be of colour 2.  A type 2 edge must have at least one vertex from $W$ and at least one vertex from $U$ which have different colours and hence the edge is not monochromatic - all other vertices in the edge are taken from $W$, $U$ and $E^*$ and hence must have colours 1 or 2.  Therefore any type 2 edge also has exactly two colours.

Thus the $k$-$(2,2)$-colouring of $H$ can be extended to a valid $k$-$(2,2)$-colouring of $H_{p,q}$.  Hence $Spec_{(2,2)}(H_{p,q})=Spec_{(2,2)}(H)$.

\end{proof}

\section{Clique-free $\sigma$-hypergraphs}

We now present some ideas and results which will then be used to present clique free $(2,2)$-colourable $\sigma$-hypergraphs with high chromatic number.

Let $\sigma = (a_1,a_2, \ldots, a_s)$  be a partition of the integer $r$ into $s$ parts, that is $\sum_{i=1}^{i=s}{a_i}=r$.  We define $\sigma_i=(a_1,a_2, \ldots,a_i-1, \ldots, a_s)$, a partition of $r-1$.  We say that $\sigma$ is \emph{symmetric} if and only if, for $1 \leq i \leq s$, all the $\sigma_i$ give the same partition of $r-1$. 

A \emph{rectangular} partition is one in which all parts are equal.  We now prove a preliminary lemma:

\begin{lemma} \label{symmetric_partition}

A partition $\sigma=(a_1,a_2, \ldots, a_s)$ is symmetric if and only if $\sigma$ is rectangular.
\end{lemma}
\begin{proof}

We first show that a rectangular partition is symmetric. Let all parts $a_1,a_2,\ldots,a_s$ be equal, say to $m \geq 1$, such that $r=sm$.  If $m=1$ (and hence $s=r$), then all $\sigma_i$ have $s-1$ parts each equal to 1, hence all $\sigma_i$'s are the same parition of $r-1$.  If $m=r$ (and hence $s=1$), then $\sigma_1 = (r-1)$, and this is the only possible deletion.  If  $1 < m <r$, then it is clear that $\sigma_i=(m,m,\ldots,m-1,\ldots,m)$, and rearranging the partitions in non-increasing order, all $\sigma_i$'s are the same partition of $r-1$.

Now consider a non-rectangular partition $\sigma$ of $r$ -- then at least two parts, say $a_i$ and $a_j$, are not equal, and $a_j > a_i \geq 1$.  If $a_i=1$ then $\sigma_i$ has $s-1$ parts and hence is different from $\sigma_j$ which has $s$ parts .  If both $a_i$ and $a_j$ are greater than 1,  then $\sigma_i$ and $\sigma_j$ both have $s$ parts, but $a_j-1 > a_i-1$ and hence the partitions are different and $\sigma$ is not symmetric.

\end{proof}

\begin{theorem}\label{rect_sigma}
Consider $H=H(n,r,q | \sigma)$.  Then 
\begin{enumerate}
\item {$H$ has an $(r+1)$-clique if and only if $\sigma=(\Delta,....,\Delta,\Delta -1)$.}
\item{$H$ has an $(r+2)$-clique if and only if either $\sigma=(r)$ and $q \geq r+2$, in  which case $H$ has a $q$-clique, or $\sigma=(1,1,\ldots,1)$ and $n \geq r+2$ in which case $H$ has an $n$-clique.}
\item{For $r \geq 4$, there exist $H$ such that $\omega(H) \leq r$.}
\end{enumerate}

\end{theorem}

\begin{proof}
1)  Consider a set of $r+1$ vertices taken from the vertex classes of $H$.  Suppose that the numbers of vertices in each class gives a partition $\sigma^*$ of $r+1$.  If these $r+1$ vertices are to form a complete subhypergraph, then the deletion of any vertex must give $r$ vertices which form the partition $\sigma$, that is  $\sigma^*$ must be symmetric and hence  by Lemma \ref{symmetric_partition} it must be rectangular and equal to $(\Delta,\Delta,\ldots,\Delta)$.  This implies that $\sigma$ must be equal to $(\Delta,\Delta,\ldots,\Delta-1)$ as claimed.

On the other hand, consider $\sigma = (\Delta,\Delta,\ldots,\Delta-1)$.  Then $r+1$ is a multiple of $\Delta$. Consider a set of $r+1$ vertices such that each set of $\Delta$ vertices is taken from a different class.  Then any $r$-subset of these vertices forms a partition $\sigma =(\Delta,\Delta,\ldots,\Delta-1)$, and hence any $r$-subset is an edge in $H$.  Thus these vertices form a clique of cardinality $r+1$.
\medskip

2)\indent Suppose $H$ has an $(r+2)$-clique on the vertex set $A$.  Let $B \subset A$, $|B|=r+1$.  Then $B$ is an $(r+1)$-clique.  Hence by (1), $\sigma=(\Delta,\Delta,\ldots,\Delta-1)$, and the vertices of $B$ are partitioned into $s$ parts with $\Delta$ vertices from each of $s$ classes of $H$.

Therefore the vertices of $A$ are partitioned into $s$ classes of $H$ to form the partition $\sigma^{**}=(\Delta+1,\Delta,\Delta,\ldots,\Delta)$ of $r+2$.  Deleting any two vertices from this set must give $r$ vertices which form the partition $\sigma$.  If s=1 then $\Delta=(r)$ and $q \geq r+2$.  If $s \geq 2$ and  $\Delta \geq 2$, then deleting two vertices from a part of size $\Delta$ in  $\sigma^{**}$ gives a partition of $r$ different from $\sigma$.  Therefore $\Delta=1$ and hence $\sigma=(1,1,\ldots,1)$ and $n \geq r+2$.

Conversely, if $\sigma=(r)$ and $q \geq r+2$,  then any $r+2$ vertices taken from the same class form an $(r+2)$-clique.  In fact, each class is a $q$-clique.  Similarly, if $\sigma=(1,1,,\ldots,1)$ and $n \geq r+2$, any $r+2$ vertices, each from a different class of $H$ form an $(r+2)$-clique. In fact, any set of $n$ vertices each taken from a different class of $H$ forms an $n$-clique.
\medskip

3)\indent For any value of $r \geq 4$, we can define a partition of $\sigma$ of $r$ different from $(\Delta,..,\Delta,\Delta-1)$.  So, for any value of $n\geq s(\sigma)$ and $r \geq 4$, there exist $\sigma$-hypergraphs, which do not contain an $(r+1)$-complete subhypergraph, and hence there exist $\sigma$-hypergraphs with $\omega(H) \leq r$.
\end{proof}

NOTE : One can observe that an $(r+1)$-clique is possible if and only if  $\Delta \mid r+1$.  The number of divisors of a positive  integer $n$, $d(n)$, clearly satisfies the inequality $d(n) \leq 2 \sqrt{n}$ (for further information, see \cite{hardy1979introduction}).  On the other hand, consider the number of different partitions of $n$, $p(n)$.  An asymptotic expression for $p(n)$ is given by \[ p(n) \sim \frac{1}{4n\sqrt{3}}e^{\pi\sqrt{\frac{2n}{3}}},\] as given in \cite{andrews1998theory}.  Hence, for large values of $r$,   almost all $\sigma$-hypergraphs are $(r+1)$-clique free.

\bigskip

Theorem \ref{rect_sigma}  can be used to show that there exist $(2,2)$-colourable sparse $\sigma$-hypergraphs which are $(r+1)$-clique free, and have arbitrarily large chromatic number.  This is another important and interesting feature of $\sigma$-hypergraphs.  

\begin{theorem} \label {cliquefreesparse}
Consider $H=H(n,r,q | \sigma = (r-1,1))$, $r \geq 4$.  For any integer $t \geq 1$,  let $n = t + 1$ and $q = (r-2)t + 1$.  Then 
\begin{enumerate}
\item{$H$ is $(2,2)$-colourable }
\item{$\chi(H) = t+1 \geq \sqrt{\frac{|V(H)|}{r-2}}$.}
\item{$H$ is $(r+1)$-clique free, that is $\omega(H) \leq r$.} 
\item{The number of edges of $H$ is  \[t(t+1)((r-2)t+1){((r-2)t+1) \choose r-1} = O(rt^{r+2}e^{r-1})\].}
\end{enumerate}
\end{theorem}

\begin{proof}
1)    We first show that there exists at least one valid $(2,2)$-colouring of $H$.  Consider $n$ colours and colour the classes monochromatically using a different colour for each class.  Then any edge has exactly $s(\sigma)=2$ colours as required.
\medskip

2) \indent  We now show that $H$ is not $t$-colourable.  Consider a colouring using $t$ colours.  Since $q=(r-2)t + 1$, there is at least one colour in each class which repeats at least $r-1$ times, and as $n \geq t+ 1$, at least one colour is repeated $r-1$ times in two or more classes.  Hence we can form a monochromatic edge from these 2 classes.  Therefore, at least $t+1$ colours are required for a valid nonmonochromatic colouring of $H$.

But in (1) we showed that $H$ is $n$-$(2,2)$-colourable, and since $n=t+1$, then $H$ is $(t+1)$-colourable.  Hence $\chi(H)=t+1$.

Now $|V(H)| = (t+1)((r-2)t+1) \leq (r-2)(t+1)^2$, and hence $(t+1) \geq \sqrt{\frac{|V(H)|}{r-2}}$.  Therefore $\chi(H) = t+1 \geq \sqrt{\frac{|V(H)|}{r-2}}$. 
\medskip

3) \indent  Since $\sigma=(r-1,1)$ and $r \geq 4$, by Theorem \ref{rect_sigma}, $\omega(H) \leq r$, and hence $H$ is $(r+1)$-clique free.
\medskip

4)\indent The number of edges is $2 {n \choose 2}{q \choose r-1}{q \choose 1} = 2 \frac{(t+1)!}{(t-1))!2!}{q \choose r-1}{q} = qt(t+1){q \choose r-1}$. Since $q = (r-2)t + 1$, this is equal to $t(t+1)((r-2)t+1){(r-2)t+1 \choose r-1}$. This is of order less than $rt^3(et)^{r-1}= rt^{r+2}e^{r-1}$, since ${n \choose k} \leq (\frac{ne}{k})^k$. 

 Hence $|E(H)| =O(rt^{r+2}e^{r-1})$.
\end{proof}

The bound in part 4 of the previous theorem can be compared with the best upper bound given in \cite{Gebauer20131483,kostochka2010}, where the authors consider the upperbound on the minimum number of edges of an $r$-uniform hypergraph with chromatic number greater than $t$, denoted by $m(r,t)$.  The upperbound given in \cite{Gebauer20131483} is $m(r,t) \leq t^{(1+o(1))r}$.

\section{Applications of $(2,2)$-spectrum preserving extensions}

In this section we present two examples  of applications of the $(p,q)$-extension.  First, we use a special case of this extension to create new $3$-uniform hypergraphs with very interesting spectra.  Then we use this same extension on $\sigma$-hypergraphs and show that the clique number remains stable at $r+1$, while the chromatic number is preserved.

Perhaps the most spectacular constructions of hypergraphs with broken spectra are those by Goinfriddo in \cite{Gionfriddo04}.  These constructions are based on $P_3$ designs and give different types of 3-uniform hypergraphs with very interesting chromatic spectra.  In Gionfriddo's colourings, every edge must have exactly two vertices of the same colour, and one vertex of another colour, and thus these are equivalent to a $(2,2)$-colouring. One of her more striking results is summarized in the following theorem:

\begin{theorem} \label{lucia}
Let $N=0 \mod 4$.  Then:
\begin{enumerate}
\item{For $N \geq 4$, there exist 3-uniform hypergraphs on $N$ vertices whose $(2,2)$-spectrum is $\{2,4,6,\ldots,\frac{N}{2} \}$.}
\item{For $N \geq 12$, there exist 3-uniform hypergraphs on $N$ vertices whose $(2,2)$-spectrum is or $\{5,7,\ldots, \frac{N-2}{2}\}$.}
\end{enumerate}
\end{theorem}

Now, using the $(p,q)$-extension, we can create new $3$-uniform hypergraphs while preserving their $(2,2)$-spectrum.  In particular, consider the $(p,q)$-extension when $W=\{w_1,w_2,\ldots,w_t\}$, $U= \emptyset$, and the set $T$ is $\{1,2,\ldots,t\}$, where $t \leq \lfloor \frac{r-1}{2} \rfloor$. We call this specific instance of the $(p,q)$-extension extension the \emph{t-star} extension and denote it by $H^*_t$.  Therefore the simplest form of the $(p,q)$-extension extension can be considered to be a $1$-star extension.  Now, if we apply the $1$-star extension once, twice or three times, we get the following corollary, which removes the restriction on $N$:

\begin{corollary}
Let $N$ be an integer.  Suppose $N=j \mod 4$.  Then there exist $3$-uniform hypergraphs on $N$ vertices with $(2,2)$-spectrum equal to
\begin{enumerate}
\item{ $\{2,4,6,\ldots,\frac{N-j}{2}\}$ for $N \geq 4$.}
\item{$\{5,7,\ldots,  \frac{N-2-j}{2}\}$ for $N \geq 12$.}
\end{enumerate}
\end{corollary}
\medskip
We now consider the application of the $t$-star extension  to an $r$-uniform $\sigma$-hypergraph $H(n,r,q \mid \sigma)$.  We study the effect of this extension on the clique number $\omega(H)$ .  The next proposition shows that the clique number is stable when the $t$-star extension is applied to an $r$-uniform hypergraph.

\begin{proposition}
Let $H$ be an $r$-uniform hypergraph.  Then \[\omega(H^*_t)= \max\{ \omega(H), r+t \}\].
\end{proposition}	
\begin{proof}
Consider the construction of $H^*_t$:  we take an edge $E^*$ composed of $r$ vertices, a set $W$ of $t$ vertices, and form all possible $r$-subsets of  $E^* \cup W$ as edges.  So this forms a clique $K$ with $r+t$ vertices.  Since $K \subset V(H)$, $\omega(H^*_t) \geq \omega(K) \geq r+t$.  Also, since no other vertex outside $E^*$ forms an edge with the vertices of $W$, then $\omega(H^*_t) = \max\{ \omega(H), r+t \}$.

\end{proof}

Now consider a $\sigma$-hypergraph $H=H(n,r,q \mid \sigma)$, with $2 \leq s(\sigma) \leq r-1$.  By Theorem \ref{rect_sigma}, $\omega(H) \leq r+1$. Hence, if we  apply the $t$-star extension, starting with $H=H(n,r,q \mid \sigma)$, recursively for any number of times, the clique number will become $r+t$ after the first extension, and remain fixed.  In particular, for a $1$-star extension, $\omega(H^*_1)=r+1$ however many times the extension is applied.  Recall that the chromatic number also remains unchanged by Theorem \ref{chrom_no}.

\section{Conclusion}

We have presented an extension of $r$-uniform hypergraphs which preserves the chromatic number and the $(2,2)$-spectrum.  The clique number also remains stable when this extension is applied repeatedly.  The most natural next step, in our view, would be devising an extension which preserves more general constrained colouring spectra, that is, $(\alpha,\beta)$-spectra with $2 \leq \alpha \leq \beta \leq r-1$. Even finding such an extension which works for $\sigma$-hypergraphs would be a significant step forward.

\bibliographystyle{plain}
\bibliography{clz2bib}

\begin{thebibliography}{10}

\bibitem{andrews1998theory}
G.~E. Andrews.
\newblock {\em The theory of partitions}, volume~2.
\newblock Cambridge University Press, 1998.

\bibitem{bujtastuz09}
C.~Bujt{\'a}s and Z.~Tuza.
\newblock Color-bounded hypergraphs, {I}: {G}eneral results.
\newblock {\em Discrete Mathematics}, 309(15):4890--4902, 2009.

\bibitem{CaroLauri14}
Y.~Caro and J.~Lauri.
\newblock Non-monochromatic non-rainbow colourings of $\sigma$-hypergraphs.
\newblock {\em Discrete Mathematics}, 318(0):96 -- 104, 2014.

\bibitem{CLZ1}
Y.~Caro, J.~Lauri, and C.~Zarb.
\newblock Constrained colouring and $\sigma$-hypergraphs.
\newblock 2014.
\newblock submitted.

\bibitem{colburn99}
C.J. Colburn, J.H. Dinitz, and A.~Rosa.
\newblock Bicolouring {S}teiner triple systems.
\newblock {\em Electronic J. Combin.}, 6:R25, 1999.

\bibitem{erdos1975problems}
P.~Erdos and L.~Lov{\'a}sz.
\newblock Problems and results on 3-chromatic hypergraphs and some related
  questions.
\newblock {\em Infinite and finite sets}, 10:609--627, 1975.

\bibitem{Gebauer20131483}
H.~Gebauer.
\newblock On the construction of 3-chromatic hypergraphs with few edges.
\newblock {\em Journal of Combinatorial Theory, Series A}, 120(7):1483 -- 1490,
  2013.

\bibitem{Gionfriddo04}
L.~Gionfriddo.
\newblock Voloshin's colourings of ${P}_3$-designs.
\newblock {\em Discrete Mathematics}, 275(1–3):137 -- 149, 2004.

\bibitem{gionfriddo04bicolouring}
M.~Gionfriddo, L.~Milazzo, A.~Rosa, and V.I. Voloshin.
\newblock Bicolouring steiner systems $s(2,4,v)$.
\newblock {\em Discrete mathematics}, 283(1):249--253, 2004.

\bibitem{hardy1979introduction}
G.~H. Hardy and E.~M. Wright.
\newblock {\em An introduction to the theory of numbers}.
\newblock Oxford University Press, 1979.

\bibitem{jiang1999}
T.~Jiang, D.~Mubayi, Z.~Tuza, V.~I. Voloshin, and D.~West.
\newblock Chromatic spectrum is broken.
\newblock {\em Electronic Notes in Discrete Mathematics}, 3:86--89, 1999.

\bibitem{kostochka2010}
A.~V. Kostochka and V.~R{\"o}dl.
\newblock Constructions of sparse uniform hypergraphs with high chromatic
  number.
\newblock {\em Random Structures \& Algorithms}, 36(1):46--56, 2010.

\bibitem{voloshin02}
V.~I. Voloshin.
\newblock {\em Coloring mixed hypergraphs: theory, algorithms and
  applications}, volume~17 of {\em Fields Institute Monograph}.
\newblock American Mathematical Society, 2002.

\bibitem{voloshin1999pseudo}
V.~I. Voloshin and H.~Zhou.
\newblock Pseudo-chordal mixed hypergraphs.
\newblock {\em Discrete mathematics}, 202(1):239--248, 1999.

\end{thebibliography}

\end{document}